\documentclass[a4paper, 11pt]{article}

\usepackage{geometry}

\usepackage{amsfonts}
\usepackage{amssymb}
\usepackage{amsmath}
\usepackage{amsthm}
\usepackage{verbatim}
\usepackage{enumitem}
\usepackage{xcolor}
\usepackage{tikz}
\usepackage{mathdots}
\usepackage{graphicx}
\usepackage{rotating}
\usetikzlibrary{calc,decorations.pathreplacing}

\def\joo#1{\cdot^{\cdot^{\cdot^{#1}}}}

\theoremstyle{plain}
\newtheorem{theorem}{Theorem}[section]

\newtheorem{proposition}[theorem]{Proposition}
\theoremstyle{definition}

\usepackage{doc}
\title{The power of clockings
}

\author{
Antti Kuusisto\\
Tampere University, University of Helsinki 
}

\date{}

\begin{document}

\maketitle

\vspace{0.5cm}

\begin{abstract}
\noindent
We investigate the expressive power of a Turing-complete logic
based on game-theoretic semantics. By defining suitable fragments and
variants of the logic, we obtain a range of natural characterizations for some
fundamental families of model classes.
\end{abstract}

\vspace{1.5cm}

\section{Introduction}

We investigate the expressive power of the Turing-complete logic
defined originally in \cite{turingcomp}. The logic is based on
game-theoretic semantics and has constructors for looping and 
modification of relations. The logic is a particularly natural 
extension of standard first-order logic $\mathrm{FO}$.

In particular, we make use of fragments and variants of the 
Turing-complete logic $\mathrm{T}$, see the
preliminaries for the specification.
The point is to 
give characterizations for some of the higher complexity 
classes. In addition to using the features already available in $\mathrm{T}$, we
make use of \emph{clocking} terms. These are syntactic operators that can be 
used to limit the time and space required in semantic games.

There exist several works that make use of clockings in the literature.
Much of the work in \cite{atl}
and \cite{atltwo} and \cite{gan}, including the background motivations for those
studies, relates to investigating the properities of $\mathrm{T}$ with clockings.
The paper \cite{atl} provides a game-theoretic semantics for $\mathrm{ATL}$ with clock values declared on the fly. The paper \cite{atltwo} has similar motivations, but investigates the more involved case of $\mathrm{ATL}^+$. The article \cite{gan}
concentrates on the $\mu$-calculus. The clockings are part of the semantics and help control the semantic game. Also non-standard clockings (i.e., ones that lead to semantic variants of the original logics) are studied with the aim of
defining interesting variants of the original systems.
Defining logics with different kinds of bounding constructors has of course been done also elsewhere, with different kinds of motivations. For some interesting examples, see, e.g., \cite{alur} and \cite{prompt}.

The clocking terms in this article are syntactic arithmetic
expressions (e.g., $2^n$) that directly limit the
duration of (or space used in) the semantic 
game. For example, if an operator is coupled with $2^n$, this means that it
should not be encountered in the
semantic game for more than $2^n$ times, where $n$ is
the size of the domain of the input model (that is, the model from
where the semantic game begins). By 
such natural and simple additions to the syntax of $\mathrm{T}$, we provide characterizations for $k$-fold exponential time and space classes for all $k$.
These are capture results in the usual sense of descriptive complexity \cite{libkin},
\cite{immerman}.

\section{Preliminaries}

We denote models by $\mathfrak{A}$, $\mathfrak{B}$, $\mathfrak{M}$, et cetera.
The corresponding domain is denoted by $A$, $B$, $M$, et cetera.
Models are assumed finite with a finite relational vocabulary. We note that we
often use the same symbol $R$ to
denote both a relation $R^{\mathfrak{M}}$ 
and the underlying relation symbol $R$. This is for the sake of simplicity.
Now, suppose we
have a linear order $<$ over the domain $M$ of a model $\mathfrak{M}$.
Suppose we have also ordered the set of
relation symbols in the vocabulary of $\mathfrak{M}$.
Then we let $\mathit{enc}(\mathfrak{M})$ denote the binary 
encoding of $\mathfrak{M}$ as defined in \cite{libkin} (please see
the full details there).
Basically, the encoding $\mathit{enc}(\mathfrak{M})$
first lists the bit $1$ for $|M|$ times,
followed by $0$. Then the relations are encoded so
that the relation encodings become concatenated one at a
time in the order that the ordering of the vocabulary requires. A $k$-ary
relation is encoded as a binary string $s$ of length $M^{k}$ 
with bit $1$ at the position $j$ of $s$ indicating
that the $j$th tuple of $M^k$ (with respect to
the lexicographic order given by $<$) is in the relation.

We assume that all relational vocabularies always have a canonical order
associated with them, so we will only have to make sure a suitable
linear is present when defining encodings of models. We note that, in the
elaborations below, we
shall \emph{in fact mainly use a successor relation} (of
the linear order) rather than a linear order itself. This obviously does not
change the encoding and is not a matter of substance any way.

In this article we study the Turing-complete logic 
defined in \cite{turingcomp}. The syntax and 
semantics of the logic is carefully specified in that article, so we
shall not reiterate all the formalities here. The main point is to add
two classes of constructs to first-order logic $\mathrm{FO}$, namely, operators
that allow to modify the underlying model and constructs that
enable looping in semantic games when evaluating formulae.
The operators that modify the model are as follows.
\begin{enumerate}
\item
The operator $Ix$ adds a new point into
the model domain and names it with the variable $x$.
So, when encountering a formula $Ix\varphi$ in a
semantic game, the following happens.
\begin{enumerate}
\item
The model is extended with a new element (keeping
all relations as they are).
\item
The new element is called $x$, i.e., the current
assignment function $g$ is modified so that it
sends the variable symbol $x$ to the new element.
\item
The game is continued from $\varphi$.
\end{enumerate}
\item
The operator $I(R(x_1,\dots , x_k))$ 
adds a tuple to the $k$-ary relation $R$ and
lets $x_1,\dots , x_k$ denote the elements of that tuple.
That is, the assignment $g$ is modified such that $x_1,\dots, x_k$ map to the 
elements of the new tuple. When encountering a formula $I(R(x_1,\dots , x_k))\varphi$,
the verifying player first does the addition and then the game 
continues from $\varphi$. Note that this step does not involve adding any new
domain points to the model. Note also that $R$ is not necessarily genuinely extended:
the verifier picks a tuple $r$ from the
model and then $R$ is updated to $R \cup \{r\}$. This may of course leave $R$ as it was, which happens if we already had $r\in R$. 
\item
The operator $D(R(x_1,\dots , x_k))$ 
deletes a tuple $(m_1,\dots , m_k)$ from the $k$-ary relation $R$ and
then the deleted tuple $(m_1,\dots , m_k)$ is marked by the
tuple $(x_1,\dots , x_k)$ of variables. That is, the
assignment $g$ is modified such that $x_1,\dots, x_k$ map to the 
elements of the new tuple.
Therefore, when encountering a formula $D(R(x_1,\dots , x_k))\varphi$,
the verifying player first does the modifications and then the game 
continues from $\varphi$. We note that the verifier is not required to
actually delete a tuple from $R$, instead the verifier appoints a 
tuple $(m_1,\dots , m_k)$ and then this tuple is
deleted from $R$ if $R$ has that tuple. Otherwise $R$ stays as it is.
Note that the domain does not become modified.
\item
The looping operator $C$ allows self-reference, i.e., formulae can
refer to themselves. The operator acts as an atomic formula as well as a
labelling operator. A formula $C\varphi$ has the \emph{label 
symbol} $C$ in front of it. Intuitively, the symbol $C$
names $\varphi$ to be the formula called ``$C$''. In a semantic game, 
from positions with $C\varphi$, we simply move on to the position with $\varphi$.
Now, $C$ can also be an atomic formula inside $\varphi$. When encountering the
atom $C$, the game jumps back to the position $C\varphi$.
When $C$ is an atom, we may refer to it as a \emph{looping atom}. For example, in the
formula $\varphi(x)\, :=\, C(Px \vee \exists y(Rxy 
\wedge \exists x(y = x \wedge C)))$ the first occurrence of $C$ is a
label symbol and the second one a looping atom. We note that non-looping-atoms are
also called \emph{first-order atoms}.
\end{enumerate}

The logic also enables the use of \emph{tape predicates}.
These are ordinary relation symbols, with the difference that
they are not included in the vocabulary of the models 
under investigation. Thus they can also be regarded as relation 
variables. The interpretation of each tape predicate is the empty 
relation in the beginning of the semantic game. During the game, alsos tape
predicates $X$ can of course be modified by the operators $D(X(x_1,\dots , x_k))$
and $I(X(x_1,\dots , x_k))$.

Now, why are the predicates $X$ that are not part of
the input signature 
called \emph{tape predicates}? The analogy
with tape symbols here is that tape predicates are
auxiliary objects used in the semantic game rather
than relations in the underlying vocabulary of
the models considered.

The arity of a relation (or the relation symbol) $R$ is
denoted by $\mathit{ar}(R)$. The same convention
holds for tape predicate symbols and the related relations.

In our logic, it is also possible to define a deletion operator $Dx$
that removes a point already labelled by $x$. Then the
assignment is modified accordingly, see
\cite{gamesandcomputation}. Whether or not we include 
this operator in our base logic does not affect the 
results below. They are the same.

The semantic games end in a position 
with a first-order atom (e.g., $R(x_1,\dots , x_k)$ or $x=y$). The verifier
wins the play if the atom holds, and otherwise the falsifier wins. If some
required move cannot be made, the game ends with neither 
player winning. For example, if a
position with a looping atom $C$ is encountered and there does not
exist a corresponding label symbol $C$ in the formula, then the game
ends with neither of the players winning. As another example, in a
first-order atom that has a variable that does not 
appear in the domain of the current assignment, the game ends
with neither player winning.\footnote{However, in the empty 
model, if a position
with $\exists x$, $I(R(x_1,\dots , x_k))$ or $D(R(x_1,\dots , x_k))$ is
encountered, then the game ends and the current verifier loses. This is
because we cannot label anything with the variable symbols.
Yet further, in the empty model, $I(R)$ and $D(R)$ for a nullary 
symbol $R$ are fine and can be performed: we can add or remove
the empty tuple  when considering a nullary relation. So there the game will  
continue. However, there is no genuinely interesting
reason to consider the empty model from the point of view of this article.}

The semantic game is played between Eloise and Abelard.
Eloise begins as the verifier (and negation changes the role of
Eloise from verifier to falsifier and vice versa).
We write $\mathfrak{M},g\models \varphi$ and consider $\varphi$ true\footnote{It may be more intuitive to consider $\varphi$ \emph{verifiable} rather than true. Technically this makes no difference.} in $\varphi$ if Eloise has a
winning strategy in the game involving $\mathfrak{M},g$ and $\varphi$. 
We may write $\mathfrak{M}\models\varphi$ if $g$ is the empty assignment or
otherwise irrelevant. We note that winning strategies are assumed 
positional. However, this makes no difference due to the positional
determinacy of reachability games (which holds
even on infinite models) \cite{paritygames}.

Recall the formula \[\varphi(x)\, :=\, C(Px \vee \exists y(Rxy 
\wedge \exists x(y = x \wedge C)))\] from
above. We have $\mathfrak{M}, \{(x,m)\} \models \varphi(x)$ iff we can 
reach from $m$ an element satisfying $P$. Note that we can make 
the semantic game always terminate in finite models by
considering instead the formula $C(Px \vee \exists y(Rxy \wedge x\not= y 
\wedge Dx\exists x(y = x \wedge C)))$.

We call the so defined logic $\mathrm{T}$. Formally, we let $\mathrm{T}$ be
the logic $\mathcal{L}$ as defined in \cite{turingcomp}. 
As discussed above, the following rules hold.
\begin{enumerate}
\item
If a first-order atom with a variable $x$ is 
reached, and the current assignment gives no interpretation to $x$,
then the play of the game ends. Neither player wins that play of the game.
\item
If a position with a looping atom $C$ is encountered, and there is no
subformula $C\varphi$ in the main formula, then the play of 
the game ends. Neither players wins that play.\footnote{We note that in \cite{turingcomp},
this rule lead to the verifier losing. This alternative convention would
not affect any of the below proofs or results.} 
\end{enumerate}
Such positions are pathological and not really 
needed for the results below. They could be avoided by
defining a suitable notion of a strongly closed formula and then limiting to
such formulae the study below.

We note that $\mathrm{T}$ does not contain $Dx$, but it
makes no difference concerning
the results and proofs below whether or not we include the
operator $Dx$ or not. We 
also note that the results and proofs fo
through as such for the logic $\mathcal{L}^*$ precisely as
defined in \cite{turingcomp}, and whether of not we add $Dx$ to $\mathcal{L}^*$ 
also makes no difference.

The formula from where a semantic game begins is often 
referred to as the \emph{original formula} or \emph{input formula}. The model in
the beginning of a semantic
game is the \emph{original model} or \emph{input model}. 
Sometimes even the terms \emph{original input formula} and \emph{original 
input model} are used.

\section{Characterizations of higher classes}

In this section we provide characterizations for
$k$-\textsc{Exptime} and $k$-\textsc{ExpSpace} 
for all $k\in\mathbb{N}$ by defining 
suitable restrictions for the looping operators $C$ and
model extension operators $Ix$. The extensions 
essentially use a term $t(n)$, depending on
the model domain size $n$, that restricts how
many times the operator instance can be used in a play of the 
semantic game. We note that for these characterizations, it makes no difference 
whether we consider ordered models or not.

Before introducing the related restriction
constructs, we begin by characterizing \textsc{ExpTime} as
the fragment of $\mathrm{T}$ that forbids the use of $Ix$.
Let $\mathrm{T}[-Ix]$ denote the restriction of $\mathrm{T}$ to
the syntax that does not allow the use of the
operator $Ix$ (for any variable $x$). Note that the logic does allow the
operators $I$ that insert tuples into relations; only the model domain 
extension capacity is restricted.

\begin{theorem}\label{exptimetheorem}
$\mathrm{T}[- Ix]$ 
captures \textsc{ExpTime}.
\end{theorem}
\begin{proof}
Fix a formula $\varphi$ of $\textrm{T}[- Ix]$. We
show how to design an \textsc{ExpTime} model checking
procedure for recognizing models of the fixed formula $\varphi$.

Now, suppose $\varphi$ has $m$
tape predicates of arities $k_1, \dots , k_m$.
Then a position in the semantic game is fully
encoded by a tuple of type $(\mathfrak{M},g,\psi,\#)$
specified as follows.
\begin{enumerate}
\item
$\mathfrak{M}$ is the underlying model at the
current stage of the game.
\item
$g$ is an assignment that interprets first order variables in
the domain $M$ of $\mathfrak{M}$. Futhermore, $g$ interprets the 
tape predicates $X_i$ as relations $X \subseteq M^{\mathit{ar}(X)}$.
\item
$\psi$ is the current subformula being played.
\item
$\#\in \{+,-\}$ indicates whether Eloise is
currently the verifying player ($+$) or 
falsifying player ($-$).
\end{enumerate}

For the fixed formula $\varphi$,
the descriptions of positions $(\mathfrak{M},g,\psi,\#)$ in the semantic game are
polynomial in the size of (the description of) the 
original input model.
Indeed, the number of tape predicates $X$ and relation symbols $R$ in
the vocabulary is constant for the fixed 
formula $\varphi$, and thus the sizes of
encodings of the corresponding relations $X\subseteq M^{\mathit{ar}(X)}$
and $R\subseteq M^{\mathit{ar}(R)}$ is
not a problem with regard to polynomiality.
Also, the number of first-order variables that
need to be encoded is constant, and so is the number of 
subformulae of the original input formula.

As we need only polynomial amount of memory to encode and
arbitary position in the semantic game, it is clear 
that we can model the semantic games for $\varphi$ (for all input
models) by an alternating polynomial space Turing
machine $\mathit{TM}_{\varphi}$. Indeed, the machine simply
keeps track of the current position, and positions where Eloise moves 
correspond to existential states while Abelard's positions correspond to
universal states.

As \textsc{APSpace}
equals \textsc{ExpTime}, we have found the required Turing 
machine corresponding to $\varphi$.

Assume then that we have a Turing machine $\mathit{TM}$ running in \textsc{APSpace}.
We will translate the alternating polynomial 
space machine $\mathit{TM}$ to a corresponding 
sentence $\varphi_{\mathit{TM}}$. Let the
space required by the machine $\mathit{TM}$ to be 
bounded from above by the polynomial $p(x)$ of order $k$.
Note that there are $|M|^{k+1}$
tuples of arity $k+1$ in a model $\mathfrak{M}$, so we can 
encode the tape cells required by $\mathit{TM}$ into $(k+1)$-tuples of the model.
However, of course we may have $p(x) > x^{k+1}$ for 
some small enough $x$, but we can deal with the those finitely many
small models by a first-order sentence $\chi$ that accepts 
precisely those small models that $\mathit{TM}$ will, rejecting the
remaining small models.
For greater sizes, we will write a separate sentence.
Thus, without loss of generality, we ignore the 
issue with small models and thereby assume that the running
time of $\mathit{TM}$ is
everywhere bounded from above by $|M|^{k+1}$.

Now, to construct $\varphi_{\mathit{TM}}$, we fix a new tape predicate $S$ which will be 
built into a successor relation over the input model.
We also define a new $(k+1)$-ary predicate $Y_q$ for each state $q$ of $\mathit{TM}$.
The computation is encoded such that
when the read-write head of $\mathit{TM}$ is in the cell $j$ and
the current state is $q$,
then the predicate $Y_{q}$ holds in the tuple $(m_1,\dots , m_{k+1})\in M^{k+1}$
that is lexicographically (with respect to the successor 
relation $S$) the $j$th tuple.
During that computation step, $Y_q$ holds nowhere else, and for
each $q'\not= q$, the predicate $Y_{q'}$ does not hold anywhere.
When modifying such a predicate $Y_q$ to simulate the computation, we can
\begin{enumerate}
\item
use an extra $(k+1)$-ary tape predicate $Y$ to encode where $Y_q$ 
currently holds,
\item
then delete $Y_q$ (i.e., to make it hold nowhere), and
\item
then use $Y$ to modify the predicate $Y_{q_{new}}$ so
that it holds at
the right slot.
\end{enumerate}
Note that the reason we use the 
temporary storage predicate $Y$ is that if $q = q_{\mathit{new}}$, 
the predicate $Y$ will help distinguishing between the old and
the new locations of $Y_q$. There are other ways around this problem,
but using the store predicate $Y$ is a particularly easy solution
suitable for our purposes.

We also define a $(k+1)$-ary predicate $X_A$ for
each (tape or input) symbol $A$ of $\mathit{TM}$.
These are modified to hold in those cells (i.e., $(k+1)$-tuples) where they would hold
during the computation.

In the beginning of computation, after
creating the successor relation $S$, the sentence $\varphi_{\mathit{TM}}$ will
make sure that the binary encoding $\mathit{enc}(\mathfrak{M})$ will be
written into the $(k+1)$-tuples of $M^{k+1}$ using 
the predicates $X_A$. This is easy to do, using
further auxiliary tape predicates $Z_i$ to gain easier control on 
the specification.

The the computation of $\mathit{TM}$ itself is
simulated in the natural way. Consider a
state $q$ and a symbol $A$.
Suppose $\mathit{TM}$ has the
following $m$ allowed transition instructions from
the state $q$ when scanning the symbol $A$:
\begin{align*}
&({q,A})\ \mapsto ({B_1},D_1,{q_1})\\
&\vdots\\ 
&({q,A})\ \mapsto ({B_m},D_m,{q_m})
\end{align*} 
where each $B_i$ denotes a new symbol to be
written to the current cell; $D_i\in\{\mathit{left},\mathit{right}\}$
denotes the direction where $\mathit{TM}$ is to move; and $q_i$
denotes the new state. For each $i$, let $\psi_{(B_i,D_i,q_m)}$
denote the formula stating that Eloise should modify 
the model as follows.
\begin{enumerate}
\item
The tape predicate $X_{B_i}$ should be made to hold in the tuple
which is at the position indicated by the store predicate $Y$ (which
points at the current position of the read-write head). Furthermore,
the tape predicate $X_A$ should be modified so that it does no 
longer hold at the position indicated by $Y$ (unless we have $A = B_i$).
All this amounts to the symbol $A$ being erased from the current cell
and $B_i$ being written to that cell instead.
\item
The tape predicate $Y_q$ should be deleted from the current cell 
and the tape predicate $Y_{q_i}$ should be made to hold in the adjacent
cell which is in the
direction $D_i\in \{\mathit{left},\mathit{right}\}$ from
the current cell.\footnote{Note that all kinds of
fringe effects are straighforward to deal with. For example, it is trivial to
alter the Turing machine so that it never attempts to go left from the 
leftmost cell.}
\item
The predicate $Y$ should be updated to hold precisely at
the position of $Y_{q_i}$ and Eloise should
enter the atom $C_{\mathit{loop}}$. However, if
the new state $q_i$ is a accepting state, then,
instead of $C_{\mathit{loop}}$, we have the atom $\top$
where Eloise can win the play of the semantic game.
Similarly, if $q_i$ is a rejecting state, then we have $\bot$ 
instead of $\top$.
\end{enumerate}
Now, if $q$ is an existential state,
then let $\varphi_{(q,A)}$
denote the formula 
\begin{align*}
\varphi_{q,A}'\, \wedge\, (\psi_{q_1}\vee \dots\ \vee \psi_{q_m})
\end{align*}
where $\varphi_{q,A}'$ states that the current 
cell where $Y_q$ and $Y$ hold has the 
symbol $A$ in it, i.e., also $X_A$ holds in that cell.
On the other hand, if $q$ is a universal state, 
then we let $\varphi_{q,A}$ be the formula
\begin{align*}
\varphi_{q,A}'\, \wedge\, \psi_{q_1}\wedge \dots\ \wedge \psi_{q_m}.
\end{align*}

Now, let $I$ list the state-symbol pairs $(q,A)$ that 
act as inputs to the transition relation that specifies $\mathit{TM}$. Define
then the formula 
\begin{align*}
\varphi_1\ :=\ C_{\mathit{loop}}\, \bigvee\limits_{(q,A)\, \in\, I}\varphi_{q,A}
\end{align*}
where we recall that the formulae $\varphi_{q,A}$ loop (when they loop) via
the looping atom $C_{\mathit{loop}}$.

Let $\alpha_{input}$ be the formula
\begin{align*} & I(X_{A_1}(x_1,\dots , x_{k+1}))C_{enc}\\
\ \vee & \\ \vdots &\\ \ \vee & \\ & I(X_{A_{m'}}(x_1,\dots , x_{k+1}))C_{enc}\\
\vee & \\
& I(Z_1(x_1,\dots , x_{\mathit{ar}(Z_1)}))C_{enc}\\
\ \vee & \\ \vdots &\\ \ \vee & \\ & I(Z_{\ell}(x_1,
\dots , x_{\mathit{ar}(Z_{\ell})}))C_{enc}\\
\end{align*}
which allows the verifier to choose one of the tape 
predicates $X_{A_i},\dots , X_{A_{m'}},
Z_1,\dots , Z_{\ell}$ and add a tuple to it, after which the
game proceeds to the looping atom $C_{enc}$. Here we 
assume that the collection of input
symbols of $\mathrm{TM}$ is $\{A_1,\dots , A_{m'}\}$. The predicates $Z_i$ are the
auxiliary predicates that help in making sure the encoding of
the input model becomes modelled correctly on the successor 
relation over the $(k+1)$-tuples of the input model $\mathfrak{M}$. Note that the successor
relation over the $(k+1)$-tuples is defined lexicographically based on the binary successor relation $S$ over the domain of the input model.
It is a relation of arity $2(k+1)$. The encoding is written to a prefix of
the related successor relation (the cells of computation).

Now, based on $\alpha_{input}$,
define the formula
\begin{align*}
\gamma\ :=\ C_{enc}((\neg\chi_{enc}\wedge \alpha_{input})
\vee (\chi_{enc} \wedge \varphi_1))
\end{align*}
where $\chi_{enc}$ states that the binary encoding of the input
model $\mathfrak{M}$ is encoded in a
prefix of the successor relation over $M^{(k+1)}$ in the correct way.
Recall that $\varphi_1$ is the formula written already above.

Define then, based on $\gamma$, the formula

\begin{align*}
\varphi_{\mathit{TM}}\ :=\ C_{succ}\bigl((\neg\chi_{succ}\wedge\ \alpha(C_{succ}))\
\vee (\chi_{succ} \wedge \gamma)\bigr)
\end{align*}
where the following conditions hold.
\begin{enumerate}
\item
$\alpha(C_{succ})$ requires Eloise to construct the 
auxiliary predicates $S$ and $S'$ so that $S'$ is a linear order and $S$
the corresponding binary successor relation over the input model.
\item
$\chi_{succ}$ is a first-order formula that states that $S'$ is a 
linear order over the model and $S$ the corresponding successor. (Including
the linear order makes it possible to
express in first-order logic that $S$ is indeed a successor.)
\end{enumerate}
Note that the main computational part of the formula is $\varphi_1$. The other 
parts relate to initial constructions creating the
binary input encoding and the successor to enabble that encoding.
\end{proof}

We then define a generalization of the syntax of the logic $\mathrm{T}$.
We begin by defining a collection of suitable terms that
relate to complexity classes with $k$-fold exponential 
limitations on resources.

Let $\mathrm{Pol}$ denote syntactic terms of the
form $$c_m n^m + c_{m-1}n^{m-1}\dots + c_1n + c_0$$ where
\begin{enumerate}
\item
$n$ is variable symbol (distinct from the usual 
logic variable symbols $x$, $y$, etc.).
\item
$c_i\in \mathbb{N}$ are
binary strings denoting natural numbers.
\item
The exponents $m-j$ are, likewise, binary strings denoting natural numbers.
\end{enumerate}
The set $\mathrm{Pol}$ thus contains (all) terms denoting
polynomials in the variable $n$ and with integer coefficients.

Now consider the set of terms of type 
$$\uparrow(k,t)$$
%
%
%
where $k$ is a binary 
string denoting a number in $\mathbb{N}$ and $t$ a term in $\mathrm{Pol}$.
Now, if $t$ denotes the polynomial $p(n)$,
then $\uparrow(k,t)$ denotes 
$$2^{2^{\joo {2^{p(n)}}}}$$
where the tower (excluding $p(n)$) is of
height $k$, i.e., the number $2$ is
written $k$ times (ignoring the 
possible occurrences of $2$ in $p(n)$).
Let $all$-\textsc{ExpTerm} denote the set of
terms of type 
\[c\, \cdot \uparrow(k,t) + d\]
where $c$, $k$ and $d$ are binary strings denoting 
numbers in $\mathbb{N}$ and $t$ is a term in $\mathrm{Pol}$. 
If $x$ and $y$ denote the natural numbers
encoded by $c$ and $d$ and if $t$
corresponds to $p(n)$, then the term $c\, \cdot \uparrow(k,t) + d$ denotes
\[x\, \cdot 2^{2^{\joo {2^{p(n)}}}} + y.\]
Note that functions in
$$O\Bigl(2^{2^{\joo {2^{p(n)}}}}\Bigr)$$
are naturally bounded by functions that
can be expressed in the form $$x\, \cdot 2^{2^{\joo {2^{p(n)}}}} + y$$
for different values of $x$ and $y$ in $\mathbb{N}$ and
with the input variable being $n$ (the tower is of course
assumed to of the same height in both cases). Thus we are encoding $k$-fold
exponential functions.
We let $k$-\textsc{ExpTerm} denote
the set of terms of type \[c\, \cdot \uparrow(\ell,t) + d\]
where $\ell$ denotes a number in $\{0,\dots , k\}$. In the case $\ell = 0$, 
the term $\uparrow(0,t)$ just outputs the term $t$. We note that 
more general sets of terms could serve our purposes quite well in
the below elaborations, but the related generalizations are not 
difficult to investigate, and the one we use here does the job
well enough from the point of view of the
current work. Generalizations are left for the future.

Suppose that $k\geq 1$. We let $\mathrm{T}[\, Ix\upharpoonright k\-\textsc{Exp}\, ]$
denote the restriction of $\mathrm{T}$ where each
operator $Ix$, $Iy$, an so on, must be written in
the form $Ix(t)$, $Iy(t)$, et cetera, where $t$ is a $k$\textsc{Exp}-term. On the
semantic side, a node $Ix(t)\phi$ in the syntax tree of the
original formula can be visited only $t$ times, i.e., the
number $\ell\in \mathbb{N}$ of
times that the term $t$ refers to. If the semantic game proceeds to that we
visit such a node for one more time, the play of the game then ends there and 
neither of the players win that play.

A natural way to think of this definition is
provided by \emph{clockings}. We add a clock function $c$ to
the semantic game. The function $c$ is similar to the assignment
function $g$ but instead gives some value $p\in \mathbb{N}$ for
each node of type $Ix(t)\psi$. Initially that value is the (binary representation of the)
number that $t$ encodes.
After that, we decrease the value for node $Ix\psi$ each time 
that the node $Ix\psi$ is visited. If we
enter the node when the \emph{clock value} is
already $0$, then the play of the game ends and indeed neither player wins that play.

\begin{theorem}\label{expspacetheorem}
$\mathrm{T}[\, Ix\upharpoonright k\-\textsc{Exp}\, ]$ 
captures $(k+1)$-\textsc{ExpTime}. 
\end{theorem}
\begin{proof}
Let $\varphi$ be a formula of $\mathrm{T}[\, Ix\upharpoonright k\-\textsc{Exp}\, ]$.
Let us find the required Turing machine.

Positions of the semantic game are of type $(\mathfrak{M},g,c,\#,\psi)$ 
such that the following conditions hold.
\begin{enumerate}
\item
$\mathfrak{M}$ is model which can be dynamically modified during the game.
\item
$g$ interprets first-order variables and tape predicates in the model.
\item
$c$ gives the clock values for the (nodes of the 
original formula) that have element insertion operators $Ix$, $Iy$, etc.
\item
$\#\in \{+,-\}$ indicates whether Eloise is the verifier ($+$) or
the falsifier ($-$).
\item
$\psi$ is a subformula of the original formula.
\end{enumerate}
Since the use of the element insertion operator is 
limited by a $k$-fold exponential term, the 
size of the (description of the) position $(\mathfrak{M},g,c,\#,\psi)$ is, likewise,
limited to $k$-fold exponential with respect to the description of the
original model. We shall do as in the proof of Theorem
\ref{exptimetheorem}, that is, we use
alternating turing machines.
Namely, we use an alternating 
Turing machine running in alternating $k$-exponential space to simulate the
evaluation game. As alternating $k$-\textsc{ExpSpace}
equals $(k+1)$-\textsc{Exptime}, the translation from the formula $\varphi$ to the
required machine is clear. We simply simulate the game with an alternating 
machine, and the space allowed for the alternating machine suffices.
For the converse translation, suppose we have an 
alternating $k$-\textsc{ExpSpace} machine $\mathit{TM}$.
Let the space required be bounded
above by the $k$-\textsc{ExpSpace} function $f(n)$ corresponding to a $k$-\textsc{ExpTerm}.
Simulating $\mathit{TM}$ by a formula $\varphi_{\mathit{TM}}$ is 
based on the following steps.
\begin{enumerate}
\item
The formula $\varphi_{\mathit{TM}}$ begins with $C\, I(t)x$ where $t$ is a
syntactic counting term for $f(n)$. In fact, we choose from $k$-\textsc{ExpTerm} a
sufficiently large term
term $t$ so that $f$ is in the big $O$ class for $t$, and we can deal with possible minor fringe effects that become realized in our below construction.

The formula $\varphi_{TM}$
enables Eloise to first add new 
points to the model, the number of them bounded by $t$. 
\item
The new points are labelled by a unary tape
predicate $N$ as `new points' not part of the
original domain of the input model.
This labelling is done simultaneously to adding the points.
The formula $\varphi_{\mathit{TM}}$ required is of the form
$$C_{build}\, I(t)x\, I(t)(Ny)\,\bigl( (C_{build} \wedge y=x) \vee \chi)$$
where $\chi$ first builds a successor relation $S$ over the new
elements and then deals with the rest of the computation. (Obviously $\chi$
does not contain looping atoms $C_{build}$, but instead uses other looping atoms.)
\item
Then a prefix (with respect to
the new successsor relation $S$ over the new elements) is labelled by 
the string $\mathit{enc}(\mathfrak{M})$, where $\mathfrak{M}$ is the
original input model. Note that the encoding
requires another successor order $S'$ to be built over $\mathfrak{M}$.
The labelling of the prefix with $\mathit{enc}(\mathfrak{M})$ is
done with unary tape predicates $P_0$ and $P_1$ for the bits $0$ and $1$.
For this step, we use auxiliary predicates (of sufficient arity) that 
scan and connect the original input model $\mathfrak{M}$ to 
the new part with the successor relation $S$. Having all of first-order logic
together with the looping capacity and a supply of auxiliary 
tape predicates in our logic, this is straightforfard to do. A similar
step was already done in the proof of Theorem \ref{exptimetheorem}, but this 
time we have a clearly separate new part of the model
where $\mathit{enc}(\mathfrak{M})$ is encoded. Furthermore, this time the bits of the
binary input can simply be unary predicates, $P_0$ and $P_1$.
\item
After $\mathit{enc}(\mathfrak{M})$ is recorded, the formula 
forces Eloise and Abelard to simulate the computation of $\mathit{TM}$ on
the new part of the model. This is done in the same way as in
the proof of Theorem \ref{exptimetheorem}. The main difference is
that now the tape contents are encoded simply by unary predicates.
\end{enumerate}
The formula $\varphi_{TM}$ thus indeed simulates $\mathit{TM}$.
\end{proof}

We then investigate the restriction of $\mathrm{T}$ to
the case where the use of looping symbols is restricted to $k$\textsc{Exp} terms.
That is, label symbols $C$ must be written in the form $C(t)$ where $t$ is a $k$\textsc{Exp}-term.\footnote{We are not attaching clock terms to looping atoms, just the corresponding label symbols. Thus it is important that in the game, when we transition from a looping atom $C$, we jump to a corresponding position $C(t)\varphi$ with the label $C$ (rather than directly to $\varphi$). Then we reduce the
corresponding clock value by one. The clock value is also lowered when we first come to the node $C(t)\varphi$ (not necessarily from a looping atom).} 
We denote this logic by $\mathrm{T}[$k$\mathrm{Exp}]$. Note that in this logic, the use of constructs $Ix$ is
trivially limited to the $k$\textsc{Exp} case because in fact all looping is limited.

Let $k\geq 1$.
We next observe that $\mathrm{T}[k\mathrm{Exp}]$
captures $k$-\textsc{ExpSpace}.

\begin{theorem}
$\mathrm{T}[$k$\mathrm{Exp}]$
captures $k$-\textsc{ExpSpace}.
\end{theorem}
\begin{proof}
The proof is almost identical to 
the proof of Theorem \ref{expspacetheorem}.
The principal difference is that this time we
limit the number of loopings (playing time) in
formulae and use the
fact that alternating $k$-\textsc{ExpTime}
equals $k$-\textsc{ExpSpace}.
\end{proof}

In general, it is interesting and worth it to use clocking terms to build
custom-made yet natural logics for capturing complexity classes. Characterizations with
polynomial clocks (terms in $\mathrm{Pol}$) are interesting of course. We 
characterize \textsc{PSpace} by the restriction of $\mathrm{T}[-Ix]$
where the looping construct $C$ is limited similarly to 
the one in $\mathrm{T}[k\textsc{Exp}]$, but this time
with the terms $t$ in $\mathrm{Pol}$.
We call this logic $\mathrm{T}[\textsc{Pol}]$.

\begin{theorem}
$\mathrm{T}[\mathrm{Pol}]$ captures \textsc{PSpace}.
\end{theorem}
\begin{proof}
The proof is almost identical to
the proof of Theorem \ref{exptimetheorem}.
This time we use the fact that
alternating polynomial time equals $\textsc{PSpace}$.
Indeed, positions in the semantic game require a polynomial 
amount of memory, and now also playing time is limited polynomially. 
Thus we can simulate semantic games with an alternating polynomial time machine.

Also the other direction is similar. The limitation of the label symbols is not a
problem when simulating an alternating machine, as the machine is now limited in
running time by a polynomial.
\end{proof}

We note that we obtain a characterization of \textsc{Elementary} for
free. The logic $\mathrm{T}[\mathit{all}\mathrm{Exp}]$ does the job.
(The logic is as $\mathrm{T}[k\mathrm{Exp}]$ but allows all
terms from $\mathit{all}$-\textsc{ExpTerm} as opposed to
only the ones in $k$-\textsc{ExpTerm}.) All kinds of clockings are of
future interest. These include term families pointing to
logarithmic (and lower) functions as well as functions that
grow extremely fast.

We also note that the above proofs do not need tape predicates, with
the exception of the characterizations of \textsc{PSpace} 
and \textsc{ExpTime}. This is because we can use different 
kinds of gadgets to encode
new relation symbols. However, for this to hold, the underlying vocabulary
must include at least one binary (or higher-arity) relation. 
Gadgets using only at most unary relation symbols do not suffice.

\subsection{Extensions}

We will then consider some extensions of the logic $T$. In this section,
when discussing the logic $T$, the reader may also consider
the extension of $T$ with all the deletion operators. The results here are
not sensitive to whether we restrict to $T$ itself of consider some of
its very close variants.

To define extensions of the logic $T$, consider first the
formulae of $T$ itself. It is easy to see that formulae with
the recursion
construct $C$ can be unraveled by replacing an atom $C$ by
the corresponding reference formula $\varphi$. The reference
formula $\varphi$ is the formula in $C\varphi$ where we use $C$ as a
naming symbol to name $\varphi$. 
Notice that here we assume that each atom $C$ has a unique reference
formula. This can be assumed without loss of generality. 
Now, we
can do the unraveling to all atoms $C$ repeatedly, ultimately 
ending up with an infinitely deep formula equivalent to
the original formula we started with.

Thus formulae of $T$ can be seen as finitary encodings of
infinitarily deep formulas. Let us define a related
infinitary language $T_{\infty}$. For this purpose,
consider first the syntax trees of formulas of $T$.
Let $V$ denote the operators that can occur in
non-leaf positions of the syntax trees of 
formulae of $T$. Thus for example $\neg$ and $\vee$ 
belong to $V$, as do $I x$ and $\forall x$, and so on. 
Similarly, let $U$ denote the labels of 
leaf positions. Thus for example $Rxy$ and $C$ are in $U$.
Now, let $V_0\subseteq V$ and $U_0\subseteq U$ be the 
restrictions of $V$ and $U$ with no looping symbols $C$,
meaning that we remove all such symbols $C_1$, $C_2$, and so on. 
The formulas of $T_{\infty}$ can now be defined as follows.

\begin{enumerate}
\item
The formulae are possibly infinite 
trees $t$ with maximum branch length $\omega$.  Like a
syntax tree, the tree is directed and has a
unique root node. 
\item
The non-leaf nodes are labeled with 
operators from $V_0$. 
\item
The leaf nodes are labeled
with atoms from $U_0$. 
\item
Nodes labeled with $\wedge$ or $\vee$ have
two children, and other nodes with a
label from $V_0$ have one child. 
\end{enumerate}

The point is that formulae of $T_{\infty}$ are like those of $T$
without looping constructs, but can have infinitely deep branches. 
The game-theoretic semantics for $T$ extends to this logic $T_{\infty}$
directly.

\begin{proposition}
Formulae of $T$ translate to formulae of $T_{\infty}$. 
\end{proposition}

\begin{proof}
The unraveling translation turns formulae in $T$ to 
ones in $T_{\infty}$. If there are atoms $C$ without 
reference formulae, these can also be dealt with using suitable
gadgets. Under the semantics where a non-referring looping atom is
considered a tie (in the game), we can create an infinite
branch with repeated, say, negations. Under the semantics
where it is a loss to one of the players, we can replace it
with a suitable formula $\top$ or $\bot$, depending on
the negations above the position. 
\end{proof}

Now, back to finitary logics, 
notice that the logic $T$ can of course be extended by a classical 
negation(s). Let $\sim$ denote the classical negation and 
extend the game-theoretic semantics of $T$ by defining 
the semantics for $\sim$ in the way described next.

The game trees are as before, the only addition being the novel 
positions where the formula begins with $\sim$. Consider (cf. \cite{turingcomp}) a 
position 
\[(\mathfrak{B},g,\#,\, {\sim\varphi}),\]
where $\mathfrak{B}$ is the current model, $g$
the assignment, $\#\in \{+,-\}$ an indicator giving
the current verifier, and $\sim\varphi$ the current formula. 
Then the next position in the game is
\[(\mathfrak{B},g,\#,\varphi).\]

This means that during a play of the evaluation game, we
simply remove $\sim$ and continue to the next position. This does
not yet determine the semantics of the extended logic. We shall 
define the semantics next, and this will
cover also the semantics of $T_{\infty}(\sim)$, that is, the
extension of $T_{\infty}$ with the possibilty of 
using $\sim$ in non-leaf positions.

Consider a formula $\varphi$, a model $\mathfrak{M}$, an
assignment $g$ and the induced game-tree beginning
from the root position $(\mathfrak{M}, g, + , \varphi)$. 
Intuitively, the positions with a formula $\sim\psi$ are at this stage 
irrelevant since nothing crucial happens in them. Thus we are
essentially considering a game for $T$ or $T_{\infty}$. 
Now, in the game-tree, label each position in the 
tree by $win(\exists)$ is Eloise has a winning strategy in
the subtree starting from that position. Similarly, 
label each position by $win(\forall)$ if Abelard has a 
winning strategy from there. Note that the game is a
reachability game for both players, so positional
strategies suffice to cover what can be
done with general strategies.

Now, as the next step, consider each node with a main 
connective $\sim$
such that the node does not have any ancestor node with 
the main connective $\sim$, that is,
consider those nodes with $\sim$ that can be reached from the
root without going through any earlier node with $\sim$. 
The nodes to be considered can be
called \emph{commencing} $\sim$-nodes. In the game tree,
replace each subtree whose root is a commencing $\sim$-node
according to the following rules. 
\begin{itemize}
\item
Suppose the node position is $(\mathfrak{M}',f,+,\sim \psi)$.
\begin{enumerate}
\item
If the node is labeled $win(\exists)$, then replace
the subtree beginning with that
node (including the node itself) by a leaf node 
with a position $(\mathfrak{M}',f,+,\bot)$.
This means that Eloise immediately loses in
such a node. 
\item
If the node is not labeled $win(\exists)$, then replace
the subtree beginning with that
node (including the node itself) by a leaf node 
with a position $(\mathfrak{M}',f,+,\top)$.
This means that Eloise immediately wins in
such a node. 
\end{enumerate}
\end{itemize}

In the same scenario, we can define an alternative
classical negation $\dot{\sim}$ as
follows.

\begin{itemize}
\item
Suppose the node position is $(\mathfrak{M}',f,+,\dot{\sim} \psi)$.
\begin{enumerate}
\item
If the node is labeled $win(\forall)$, then replace
the subtree beginning with that
node (including the node itself) by a leaf node 
with a position $(\mathfrak{M}',f,+,\top)$.
This means that Eloise immediately wins in
such a node. 
\item
Otherwise replace
the subtree beginning with that
node (including the node itself) by a leaf node 
with a position $(\mathfrak{M}',f,+,\bot)$.
This means that Eloise immediately loses in
such a node. 
\end{enumerate}
\end{itemize}

The first classical negation has the reading ``is not true''
and the second one ``is false.'' (Alternatively, we can
read these as ``is not verifiable'' and ``is
falsifiable.'') Note that we do not consider here 
game-trees where $\neg$ can occur before $\sim$ or $\dot{\sim}$. As
particularly interesting cases, we mention the 
game-theoretically fully dual ones.

To define truth (or a related 
concept), we write $\mathfrak{M}\models^+ \varphi$ if
Eloise has a winning strategy in the game involving $\varphi$. If
the formula involves $\sim$ or $\dot{\sim}$, then the 
game is of course the two-stage game where we first consider
all strategies, then label nodes, and then consider playing in
the new, modified game-tree. We also write $\models^+ \varphi$ if we
have $\mathfrak{M}\models^+ \varphi$ for all models $\mathfrak{M}$.

Now, consider the sentences $C\neg C$ and $C \sim C$ and
also $C \dot{\sim} C$. 
These give the liar paradox, or formalizations of it. Let us consider, in
particular, $C \sim C$. Now, let us indeed give this the above
semantics. We first observe that neither $\exists$ nor $\forall$ has a
winning strategy in the game. Thus, according to the above
semantics, the commencing $\sim$-position does
not become labeled with either $win(\exists)$ or $win(\forall)$.
Therefore we have $\models^+ C\sim C$. Note also
that we have $\mathfrak{M}\not\models^+ C\, \dot{\sim}\, C$
for all models $\mathfrak{M}$.

Now, consider $C\neg C$ again. This formula is
indeterminate, neither player has a winning strategy. 
This means that we do not get a well-founded truth
for $C\neg C$. Here, a well-founded truth means Eloise
winning the standard evaluation game for $\varphi$. This is
the notion of truth (or verifiability) in $T$ (cf.
\cite{gamesandcomputation}).
Under that notion, truth means that a 
well-founded procedure exists 
for reducing the truth of $\varphi$ to truths of
first-order atomic formulas.\footnote{We remark
that in finite models, even winning strategies are 
finite in $T$ (due to K\"{o}nig's lemma).} A
winning strategy gives a well-founded subtree of
the full game-tree where every path leads in a finite
number of steps to an obvious atomic truth. This subtree
can be seen as a proof of $\varphi$ in the relevant model. In
general, this gives a natural pre-theoretic concept for 
truth: it needs to be based on a well-founded, reductionist 
and finite (or somehow finitary) process that ends with
atomic facts whose truth value is self-evident. The key is 
\emph{reduction} to 
atomic literals with obvious truth values. This could be
described as a coinductive reduction process.

Concerning general intuitions about truth, we often 
have such a well-founded procedure in mind. In the 
liar sentence $C \neg C$ and truth teller $C C$, the
related procedure fails (cf. \cite{gamesandcomputation}). The
attempt to find the firm ground by working towards
atomic truths fails, and instead the process seems to 
run infinitely long. Thus, for a suggested resolution of
the paradox, we can consider the following.

Firstly, let $LS$ 
denote the liar sentence, which we do not have to fix
syntactically or semantically here. Nevertheless, $C\neg C$ gives one
possibility, but of course not the only one. 
Similarly, let $TT$ denote the truth teller.

Now, as a potential strategy for explaining the 
paradox, we first accept that neither $LS$ nor $TT$ has a
well-founded truth value (or a truth value based on a
well-founded process) in the reductionist
sense described above. We equate such well-founded
truth values with the standard, desired and
unquestioned truth. We could characterize it as
first-level truth.

As the next step, we may require that some
truth value must be found for $LS$ (or $TT$ or both). 
This steps seems questionable, as it rests upon
requiring bivalence. However, if we in any case
wish to force a truth value for $LS$ or $TT$, we 
can do so without any problems. This is because what we are
after is not forcing a first-level truth. Instead, we
seek a second-level one. There is, of course, a conceptual similarity 
here with Russell's type theoretic hierarchies. Note
that this move does not
depend on the fact that we chose our first-level truth to be
the well-founded truth. We could choose different notions of
truth as our first-level truths, or even simply not point out
any. The important issue is to declare that whatever the first-level
truth is, we are now doing something beyond that. Indeed, it is an
important point that while our first-level truth here relates to
well-foundedness, we could do the argument without specifying what
first-level truth is.

This approach clarifies things. We are not forcing the same,
first-level truth notions on $LS$ and $TT$, we are choosing 
higher-level ones. They need not be comparable with the 
first-level ones. Furthermore, we can choose them in any 
arbitrary way without being inconsistent. This is not to say
that all ways are equally natural, but they are consistent.

Which values to choose then? This can be done arbitrarily, as  
the second-level truth does not interfere with the first-level one. 
This indeed does not mean that all choices are equally natural in 
every possible way, but it 
does mean that no choice is inconsistent. A short path to
the same conclusion goes as follows. Firstly, $LS$ seems to give an
infinitely flipping sequence of (first-level) truths, true-false-true,
and so on. Therefore the (second-level) truth-value of $LS$ is ``infinitely
flipping (first-level) truth values.''

To summarize, the problem was to expect a first-level reductionist
truth value. That process did not stop. The next step was to try to 
force a truth or falsity with the direct reading of the sentence. This
seemed to lead to a flipping truth value. The next step was to force a
second-level truth value, which ever one. This is analogous to adding
imaginary numbers to the reals. In fact, paraconsistent truth values are
also easy to accept with a similar abstraction, considering them just 
new abstract entities with a new kind of an interpretation that 
expands the old paradigm.

Now, back to the formal semantics. Notice that there we also 
found different truth values for the paradoxical
sentence ``$C \text{ negation } C$'',
one for the case where $\sim$ is the negation and another one for $\dot{\sim}$. 
However, a perhaps even more natural option would simply be to give it
the new truth value ``infinitely flipping'' which does not 
try to get associated with true or false too directly, and 
which also describes what happens with the 
attempt to get reductionist first-level truth values.

Finally, the logic $T$ obtains a natural compositional 
semantics as a corollary of the game-theoretic one. However, the
semantics has an interesting issue. For the first-order
connectives, the semantics goes as follows.

\[
\begin{array}{ll}
\mathfrak{A},g \models^+ \varphi \wedge \psi &
\ \Leftrightarrow\ \ \
\mathfrak{A},g \models^+ \varphi \text{ and } 
\mathfrak{A},g \models^+ \psi\\
\mathfrak{A},g \models^+ \varphi \vee \psi &
\ \Leftrightarrow\ \ \
\mathfrak{A},g \models^+ \varphi \text{ or } 
\mathfrak{A},g \models^+ \psi\\
\mathfrak{A},g \models^+ \neg \varphi &
\ \Leftrightarrow\ \ \
\mathfrak{A},g \models^- \varphi\\
\mathfrak{A},g \models^+ \exists x\varphi &
\ \Leftrightarrow\ \ \
\mathfrak{A},g[a/x] \models^+ \varphi\text{ for some }a\in A\\
\mathfrak{A},g \models^+ \forall x\varphi &
\ \Leftrightarrow\ \ \
\mathfrak{A},g[a/x] \models^+ \varphi\text{ for all }a\in A\\
\mathfrak{A},g \models^- \varphi \wedge \psi &
\ \Leftrightarrow\ \ \
\mathfrak{A},g \models^- \varphi \text{ or } 
\mathfrak{A},g \models^- \psi\\
\mathfrak{A},g \models^- \varphi \vee \psi &
\ \Leftrightarrow\ \ \
\mathfrak{A},g \models^- \varphi \text{ and } 
\mathfrak{A},g \models^- \psi\\
\mathfrak{A},g \models^- \neg \varphi &
\ \Leftrightarrow\ \ \
\mathfrak{A},g \models^+ \varphi\\
\mathfrak{A},g \models^- \exists x\varphi &
\ \Leftrightarrow\ \ \
\mathfrak{A},g[a/x] \models^- \varphi\text{ for all }a\in A\\
\mathfrak{A},g \models^- \forall x\varphi &
\ \Leftrightarrow\ \ \
\mathfrak{A},g[a/x] \models^- \varphi\text{ for some }a\in A\\
\end{array}
\]

The clauses for first-order atomic formulae are as
usual. This semantic system extends to $T$ easily. First note that we
have the following.

\[
\begin{array}{ll}
\mathfrak{A},g \models^+ Ix \varphi &
\ \Leftrightarrow\ \ \
(\mathfrak{A} + a),g[a/x] \models^+ \varphi\\
\mathfrak{A},g \models^+ I(Rx_1...x_k) \varphi &
\ \Leftrightarrow\ \ \
\mathfrak{A},g[(R \cup (a_1,..., a_k)) /R][(a_1,... , a_k) / 
(x_1,... , x_k)] \models^+ \varphi \\
& \text{ } \hspace{1.4cm} \text{ for some }a_1,\dots a_k \in A\\
\mathfrak{A},g \models^- Ix \varphi &
\ \Leftrightarrow\ \ \
(\mathfrak{A} + a),g[a/x] \models^- \varphi\\
\mathfrak{A},g \models^- I(Rx_1...x_k) \varphi &
\ \Leftrightarrow\ \ \
\mathfrak{A},g[(R \cup (a_1,..., a_k)) /R][(a_1,... , a_k) / 
(x_1,... , x_k)] \models^- \varphi \\
& \text{ } \hspace{1.4cm} \text{ for all }a_1,\dots a_k \in A
\end{array}
\]

\noindent 
where $\mathfrak{A} + a$ means $\mathfrak{A}$ 
expanded with a fresh isolated element $a$. 
The clauses for the deletion operators are similar and easily understood, so we
skip them. The clauses for $C$ are as follows.

\[
\begin{array}{ll}
\mathfrak{A},g \models^+ C \varphi &
\ \Leftrightarrow\ \ \
\mathfrak{A},g \models^+ \varphi\\
\mathfrak{A},g \models^+ C &
\ \Leftrightarrow\ \ \
\mathfrak{A},g \models^+ C \varphi\\
\mathfrak{A},g \models^- C \varphi &
\ \Leftrightarrow\ \ \
\mathfrak{A},g \models^- \varphi\\
\mathfrak{A},g \models^- C &
\ \Leftrightarrow\ \ \
\mathfrak{A},g \models^- C \varphi \\
\end{array}
\]
where we assume $C$ has a unique reference formula.
If not, we need to take into account many reference 
formulas, which is also easy to formulate.

We note that these clauses (the compositional semantics) is, in 
some sense, a corollary of the game-theoretic semantics. Also, we note that we
will not get a clear evaluation of (for example) the sentence $C C$ by 
using this compositional semantics. The game-theoretic semantics does tell
that the sentence is indeterminate, but the compositional clauses we 
deduced here do not say anything. They simply keep 
referring to each other. Nevertheless, the above compositional semantic
equivalences are true. The possible circularity is not really an issue, and of course 
circularity does not imply inconsistency anyway, just that the cicrular 
point is underdetermined.

\bibliographystyle{plain}
\bibliography{mybib}


\end{document}